\documentclass[10pt]{amsart}
\usepackage{amsthm,amsmath,amssymb,amscd,graphics,enumerate, stmaryrd,xspace,verbatim, epic, eepic,url}
\usepackage[usenames,dvipsnames]{color}

\usepackage[colorlinks=true]{hyperref}

\usepackage[active]{srcltx}
\usepackage[all]{xypic}
\usepackage{mathrsfs}
\usepackage{tikz-cd}

\def\AF{{\mathcal{AF}}}
\def\sat{{\rm sat}}
\def\:{\colon}
\def\.{,\dots,}
\def\into{\hookrightarrow}
\newcommand{\oM}{{\overline{M}}}
\newcommand{\oP}{{\overline{P}}}
\newcommand{\oQ}{{\overline{Q}}}
\newcommand{\oR}{{\overline{R}}}
\def\Aut{{\rm Aut}}
\def\toisom{\xrightarrow{{_\sim}}}

\def\ZZ{\mathbf Z}

\def\NN{\mathbf N}

\def\cL{\mathcal L}
\def\cO{\mathcal O}
\def\uX{{\underline{X}}}
\def\uS{{\underline{S}}}

\def\L#1{\mathcal{L}og_{#1}}
\def\u#1{\underline{#1}}
\def\o#1{\overline{#1}}
\def\gp{\textrm{gp}}
\def\c#1{\mathcal{#1}}
\newcommand{\cF}{{\mathcal F}}
\newcommand{\cZ}{{\mathcal Z}}
\newcommand{\cX}{{\mathcal X}}

\newcommand{\cU}{{\mathcal U}}

\newcommand{\HOM}{\operatorname{HOM}}
\newcommand{\Hom}{\operatorname{Hom}}
\newcommand{\Spec}{\operatorname{Spec}}

\newcommand{\dirlim}{\displaystyle \lim_{ \longrightarrow } \,}

{
   \newtheorem{theorem}[subsubsection]{Theorem}
      \newtheorem*{theorem*}{Theorem}
   \newtheorem{proposition}[subsubsection]{Proposition}

   \newtheorem{lemma}[subsubsection]{Lemma}

   \newtheorem{corollary}[subsubsection]{Corollary}
   
   \newtheorem*{conjecture*}{Conjecture}

}
{\theoremstyle{definition}
          \newtheorem*{exercise*}{Exercise}
   
   \newtheorem{example}[subsubsection]{Example}
   \newtheorem*{example*}{Example}
   \newtheorem{definition}[subsubsection]{Definition}
   
   \newtheorem*{definition*}{Definition}
   
   \newtheorem{remark}[subsubsection]{Remark}

}

\title{Logarithmically Regular Morphisms}
\begin{document}

\author[S. Molcho]{Sam Molcho}
\address{Einstein Institute of Mathematics\\
               The Hebrew University of Jerusalem\\
                Giv'at Ram, Jerusalem, 91904, Israel}
\email{samouil.molcho@math.huji.ac.il}

\author[M. Temkin]{Michael Temkin}
\address{Einstein Institute of Mathematics\\
               The Hebrew University of Jerusalem\\
                Giv'at Ram, Jerusalem, 91904, Israel}
\email{temkin@math.huji.ac.il}

\thanks{This research is supported by ERC Consolidator Grant 770922 - BirNonArchGeom}
\thanks{The project has received funding from the European Research Council (ERC) under the European Union Horizon 2020 research and innovation program (grant agreement No.786580).}
\thanks{We would like to thank the anonymous referee for valuable comments.}

\date{\today}

\bibliographystyle{amsalpha}
\maketitle

\begin{abstract}
We consider the stack $\L X$ parametrizing log schemes over a log scheme $X$, and weak and strong properties of log morphisms via $\L X$, as defined by Olsson. We give a concrete combinatorial presentation of $\L X$, and prove a simple criterion of when weak and strong properties of log morphisms coincide. We then apply this result to the study of logarithmic regularity, derive its main properties, and give a chart criterion analogous to Kato's chart criterion of logarithmic smoothness.
\end{abstract}

\section{Introduction}

\subsection{Olsson's approach to log geometry}
In \cite{Ostack}, Olsson defined and studied the stack $\L X$ of logarithmic schemes over a given fine logarithmic scheme $X$: Objects of $\L X$ are fine log schemes over $X$ and morphisms between two fine log schemes are morphisms over $X$ that are strict, that is, which preserve the log structure. One of Olsson's key insights was to apply the construction of $\L X$ to the study of morphisms of logarithmic schemes. According to Olsson, a morphism $f:X \rightarrow Y$ of fine log schemes should have a ``log'' property $\c{P}$ if the associated morphism $\L f : \L X \rightarrow \L Y$ has the same property $\c{P}$. There is however a second natural possible definition, in which one requests that the composed morphism $X \rightarrow \L X \rightarrow \L Y$ has property $\c{P}$. Morphisms $f$ that satisfy this second definition are said to satisfy ``weak log'' property $\c{P}$. In general, the two notions do not coincide, and each has its distinct set of advantages. The strong property has better functorial behavior, but is much harder to work with explicitly. Olsson showed in \cite{Ostack} that the strong and weak properties coincide for three important classes of morphisms, namely log \'etale, log smooth, and log flat morphisms. All three classes had been introduced before by Kato, the first two in \cite{Klog}, and log flatness in an unpublished article. The definitions differ in nature however. Log \'etaleness and log smoothness were defined intrinsically, via a logarithmic analogue of the classical infinitesimal lifting property that defines \'etale and smooth morphisms. Then, Kato proved an extremely useful ``chart'' criterion which characterizes the local behavior of log smooth and \'etale maps. On the other hand, logarithmic flatness did not possess an intrinsic definition, and was defined by using a modification of the chart criterion as its definition. Olsson's approach thus unified the three classes of morphisms under the same principle, and provided evidence that his definition is the ``correct'' one, at least when a weak and strong log property coincide.

\subsection{Main results}
Olsson's definition of $\L X$, and the derivations of its main properties, are rather abstract -- this is also one of the reasons that the morphism $\L X \rightarrow \L Y$ can be difficult to work with. In this paper, we provide a simplicial presentation of $\L X$. This presentation is in terms of concrete combinatorial objects, constructed from the fan of the log scheme $X$. More precisely, when $X$ posesses a global chart, we construct a stack $\c{L}_X$ which is entirely combinatorial in nature from this chart, and show

\newtheorem*{theorem: logXth}{Theorem \ref{logXth}}
\begin{theorem: logXth}
If $X$ is a log scheme possessing a global chart, $\cL_X$ is canonically equivalent to $\L X$.
\end{theorem: logXth}

We then show how to globalize this construction to an arbitrary $X$ -- but the upshot is that the local structure of $\L X$ is made completely explicit, and gives some intuition, if only to the authors. We then proceed to use the local structure of $\L X$ to give a simple criterion for comparing strong and weak log properties. Explicitly, we show

\newtheorem*{theorem: logPiffwlogP}{Theorem \ref{theorem: logPiffwlogP}}
\begin{theorem: logPiffwlogP}
Suppose $\c P$ is a property of representable morphisms of stacks which is stable under pullbacks and can be checked \'etale locally on the source and flat locally on the target. Then log $\c P$ and weak log $\c P$ are equivalent.
\end{theorem: logPiffwlogP}

This brings us to our main motivation for writing this paper: the notion of logarithmic regularity. This notion is extensively used in \cite{relative} in constructing a theory of resolution of morphisms and generalizes the usual logarithmic smoothness to the case of morphisms not of finite type. Until now, only the absolute notion of logarithmic regularity has been studied, that is to say, what it means for a scheme to be logarithmically regular. This was done  in \cite{Ktor} for fine saturated log schemes and in \cite{GabberRamero} for fine log schemes. Olsson's approach provides two natural candidates for a definition, and given theorem \ref{theorem: logPiffwlogP}, we can immediately deduce that the two possible definitions coincide. From this, in section \ref{section: logreg} we deduce that the notion of log regularity has all the basic properties one might expect. Finally, we prove a chart criterion for log regular morphisms, analogous to Kato's chart criterion for log smooth morphisms.

\newtheorem*{theorem: Chart Criterion}{Theorem \ref{theorem: Chart Criterion}}
\begin{theorem: Chart Criterion}
A map $f:X \rightarrow Y$ of locally Noetherian log schemes is log regular if and only if, for any fppf chart $Q \rightarrow M_Y$, \'etale locally around any point $x$ in $X$, there exists an injective chart $Q \rightarrow P$ for the morphism $f$ such the torsion part of the cokernel of $Q^{\rm{gp}} \rightarrow P^{\rm{gp}}$ has order invertible at $x$, and such that the morphism $X \rightarrow Y\times_{\Spec \ZZ[Q]}\Spec \ZZ[P]$ is regular.
\end{theorem: Chart Criterion}

We in fact formulate and prove a slightly sharper version in \ref{section: Chart Criterion}. From this chart criterion, we can easily deduce that this relative notion of logarithmic regularity extends the absolute notions in the literature.

\subsection{Conventions and notations}\label{convsec}
In this paper we denote log schemes by single letters such as $X,Y,Z$ etc. We denote the underlying scheme of $X$ by $\u X$ and the sheaf of monoids by $M_{X}$. We also denote the structure sheaf of $\u X$ as $\c{O}_X$ instead of $\c{O}_{\u{X}}$. Similarily, the morphism underlying a morphism $f:X \rightarrow Y$ of log schemes is denoted by $\u f$; the morphism of sheaves of monoids is usually denoted by $f^\flat$ or simply as an arrow $f^*M_Y \rightarrow M_X$. The group of units of a monoid $M$ is denoted by $M^*$, and the quotient $M/M^*$, called the sharpening of $M$, by $\o{M}$. A monoid $M$ is called sharp if $M = \o{M}$. We call the sharpening $\o M_X = M_X/\mathcal{O}_X^*$ of a log structure the ``characteristic monoid'' of $X$, and $M_{X/Y} = \o M_{X/Y} = M_X/f^*M_Y$ the relative characteristic monoid to conform with the literature.

\section{Fans and toric stacks}

\subsection{Fans}
We will always work with fine monoids.

\subsubsection{Recollections}
Recall that to any fine monoid $P$ one can associate an {\em affine fan} $F=\Spec(P)$, whose points are prime ideals, the Zariski topology is generated by {\em face embeddings} $\Spec(P[-m])\into\Spec(P)$ and the stalk of the structure sheaf $M_F$ at $p$ is the sharpened localization of $M$ at $p$. Clearly, $\Spec(F)$ depends only on the {\em sharpening} $\oP= P/P^\times$. By a {\em fan} or {\em Kato fan} we mean a monoidal space $(X,M_X)$, which is locally isomorphic to an affine fan. A fan is called {\em saturated} if it is covered by $\Spec(P_i)$ with fs $P_i$, and the theory of saturated fans is equivalent to the classical theory of rational polyhedral (or toric) fans.

\subsubsection{\'Etale morphisms}
An {\em \'etale morphism} of fans is just a local isomorphism, so we will later introduce stacks using the usual Zariski topology on the category $\AF$ of affine fans.

\subsubsection{Generalized cone complexes}
In the recent work \cite{ACP}, Abramovich, Caporaso and Payne constructed a tropical skeleton of the moduli space of curves $\oM_{g,n}$ analogous to skeletons of toroidal varieties. Not surprisingly, this construction required to introduce an analogue of stacks in the category of toric fans, that they called generalized cone complexes. Informally, these are colimits of diagrams of affine saturated fans with all maps being face embeddings. If the diagram is a poset, then one ends up with a usual fan. In the case of toric fans, the theory can be developed using topological realizations and their actual colimits.

\begin{example}
\label{satuniversalfan}
Taking $\cL_\sat$ to be the colimit of the diagram of all affine toric fans and all face embeddings between them one obtains a huge generalized complex such that any toric fan $F$ possesses a unique \'etale morphism $F\to\cL_\sat$. In particular, $\cL_\sat$ is the 2-final object in the 2-category of stacky fans with \'etale morphisms. The same construction applies to the category of all affine fans over a fixed fan $X$, and the obtained object will be denoted $\cL_{X,\sat}$.
\end{example}

\subsubsection{Stacky fans}
We would like to extend the above example to the category of all fans. First, instead of generalized complexes we will work with stacky fans, which are simple analogues of algebraic stacks. A stack $\cX$ on the category $\AF$ with Zariski topology is called a {\em stacky fan} if its diagonal is representable and there is an \'etale cover $X_0\to\cX$ with $X_0$ a functor represented by a fan. As usually, setting $X_{n+1}=X_n\times_\cX\cX_0$ one obtains an \'etale groupoid in fans, and the 2-category of stacky fans can be also described (up to an equivalence) as the 2-category of \'etale groupoids in fans. This definition of stacky fans has appeared and been studied in \cite{UliAF}.

\begin{remark}
Providing the full simplicial fan $X_\bullet$ is a matter of taste. It is completely determined by the usual groupoid datum $X_1\rightrightarrows X_0$, $m\:X_2\to X_1$, etc. We prefer to write $X_\bullet$ instead of $X_1\rightrightarrows X_0$ because stating associativity of $m$ requires to provide some information up to level $X_3$.
\end{remark}

\subsubsection{Joint faces and join fans}
Once the necessary framework has been fixed, we can try to construct the universal $X$-fan $\cL_X$ in the category of fans with \'etale morphisms. Instead of trying to define it as a (homotopy?) colimit, we will simply write down an explicit groupoid presentation. By a {\em joint face} or simply a {\em face} of affine fans $X_i=\Spec(Q_i)$, $0\le i\le n$ we mean an isomorphism class of diagrams $h_i\:F\into X_i$, where each $h_i$ is a face embedding. Clearly, all joint faces form a poset that we denote $\cF(Q_0\.Q_n)$, and hence the colimit of all faces $F$ over the diagram $\cF(Q_0\.Q_n)$ is a {\em join fan}, that will be denoted $J(Q_0\.Q_n)$. In particular, $J(Q)=\Spec(Q)$.

\begin{lemma}\label{facelem}
If $Q_0\.Q_n$ are affine fans and $0\le l\le n$, then $$J(Q_0\.Q_n)=J(Q_0\.Q_l)\times_{\Spec(Q_l)}J(Q_l\.Q_n).$$
\end{lemma}
\begin{proof}
Set $X_i=\Spec(Q_i)$. The righthand side is a colimit over the product poset of faces $F=F_1\times_{X_i}F_2$, where $F_1$ is a joint face of $X_0\.X_l$ and $F_2$ is a joint face of $X_l\.X_n$. Since $F$ is a joint face of $X_0\.X_n$ and any joint face of $X_0\.X_n$ is obtained at least once as $F_1\times_{X_l}F_2$, we obtain the required isomorphism.
\end{proof}

\subsubsection{The universal fan}
Now let $X$ be a fan. For any $n$ let $L_n=\coprod_QJ(Q_0\.Q_n)$ be the disjoint union of joins of all $n+1$-tuples of affine fans $\Spec(Q_i)$ over $X$. The natural restriction and degeneracy maps from $J(Q_0\.Q_n)$ to the joins $J(Q_0\.Q_{i-1},Q_{i+1}\.Q_n)$ and $J(Q_0\.Q_{i}, Q_{i}\.Q_n)$, respectively, give rise to a simplicial fan $L_\bullet$ over $X$.

\begin{theorem}\label{universalfan}
Let $X$ be a fan. The associated simplicial fan $L_\bullet$ is a groupoid and the quotient $\cL_X$ is a stacky fan, which is the 2-final object in the 2-category $\mathcal{F}_X$ of stacky fans over $X$ with \'etale moprhisms. 
\end{theorem}
\begin{proof}
It follows from Lemma~\ref{facelem} that $L_n=L_1\times_{s,L_0,t}\dots\times_{s,L_0,t}L_1$, where $s,t$ are the two maps $L_1\to L_0$. To conclude that $L_\bullet$ is a groupoid it remains to check the associativity constraint, which follows from the fact that the compositions $J(Q_0,Q_1,Q_2,Q_3)\to J(Q_0,Q_1,Q_3)\to J(Q_0,Q_3)$ and $J(Q_0,Q_1,Q_2,Q_3)\to J(Q_0,Q_2,Q_3)\to J(Q_0,Q_3)$ coincide.

It remains to prove that any affine fan $F$ over $X$ admits a unique \'etale morphism $f\:F\to\cL_X$, in particular, $f$ has no automorphisms. The groupoid $L_\bullet$ is \'etale, hence $L_0\to\cL_X$ is \'etale. Any morphism from an affine fan lifts through \'etale morphisms of fans, hence also of stacky fans. In particular, any \'etale morphism $F\to\cL_X$ factors through an \'etale morphism $F\to L_0$. So, we should prove that for any pair of \'etale morphisms $f,g\:F\to L_0$ there exists a unique morphism $h\:F\to L_1$ such that $s\circ h=f$ and $t\circ h=g$. It remains to note that $f$ and $g$ correspond to face embeddings $F\into\Spec(Q_0)$ and $F\into\Spec(Q_1)$, and the unique $h$ is the joint face $F\into J(Q_0,Q_1)$.
\end{proof}

\begin{remark}
Even though the fan $X$ itself belongs to the category $\mathcal{F}_X$, it is not its terminal object, as the structure map $F \rightarrow X$ from an arbitrary stacky fan $F$ is not necessarily \'etale. The stacky fan $\cL_X$ is enormous; it is the analogue of the cone complex $\mathcal{L}_{\textup{sat}}$ of example \ref{satuniversalfan} in the category of all fine (rather than fs) monoids over $X$. In fact, $\mathcal{L}_\textup{sat}$ is the generalized cone complex corresponding to the saturation of $\cL_X$. 
The stacky fan $\cL_X$ is the combinatorial analogue of Olsson's stack $\L X$ discussed in the next section. This analogy will be made precise in section \ref{section: logproperties}.
\end{remark}

\subsection{Relative toric stacks}\label{reltoricsec}
The goal of \S\ref{reltoricsec} is to straightforwardly extend the theory of toric stacks to the relative case over a log scheme $X$ provided with a global chart $X\to\Spec(\ZZ[P])$. By $P\to Q$ we always denote a homomorphism of fine monoids.

\subsubsection{The affine case}
Given a monoid $P$ consider the absolute toric scheme $Z_P=\Spec(\ZZ[P])$ with the corresponding torus $T_P=Z_{P^\gp}$ and toric stack $\cZ_P=[Z_P/T_P]$. Then we introduce the corresponding {\em relative toric $X$-scheme} $X_P[Q]=X\times_{Z_P}Z_Q$ and the {\em relative toric $X$-stack} $\cX_P[Q]=X\times_{\cZ_P}\cZ_Q$.

\subsubsection{Basic properties}
The following facts about toric stacks are very easy, so we skip the justification:

(1) $\cX_P[Q]$ is the quotient of $X_P[Q]$ by the $X$-torus $X\times T_{Q/P}$, where $T_{Q/P}:=T_Q/T_P$.

(2) $\cX_P[Q]$ depends only on the sharpenings $X\to Z_\oP$ and $\oP\to\oQ$.

(3) For any localization $R=Q[-m]$, the homomorphism $\cX_P[\oR]=\cX_P[R]\to\cX_P[Q]$ is an open immersion. We call it a {\em face map} because it corresponds to the face embedding $\Spec(\oR)\into\Spec(Q)$.

\subsubsection{The log structure}
For brevity, set $\cX=\cX_P[Q]$, $X_0=X_P[Q]$ and $X_1=X_0\times_\cX X_0=X_0\times T_{Q/P}$. It is easy to see that the two pullbacks of the log structure on $X_0$ with respect to the projection and the action maps $X_1\rightrightarrows X_0$ coincide; in particular, we can safely view $X_1$ as a log scheme, and the log structure descends to a log structure $M_\cX$ on $\cX$. Informally, the latter can be viewed as an avatar of the sharpening $\oM_{X_0}= M_{X_0}/\mathcal{O}_{X_0}^\times$. In particular, this is indicated by the following result (see also \cite[Proposition 5.17]{Ostack} and \cite[Lemma~6.3]{CCUW}).

\begin{lemma}\label{toriclem}
Keep notation of \S\ref{reltoricsec}. Then the log scheme $X_P[Q]$ and the log stack $\cX_P[Q]$ represent the functors on the category of log schemes over $X$, which send $S$ to the homomorphisms of monoids

\begin{align*}
\Hom_P{(Q, \Gamma(S,M_S))} \ \ \ {\rm and} \ \ \ \Hom_P{(Q, \Gamma(S,\o{M}_S))}.
\end{align*}
\end{lemma}
\begin{proof}
The first claim is obvious, and it implies that the functor represented by $\cX_P[Q]=[X_P[Q]/X_P[Q^\gp]]$ is the quotient of $\Hom_P(Q, \Gamma(M_S))$ by $$\Hom_P(Q^\gp, \Gamma(M_S))=\Hom_P(Q, \Gamma(M^\times_S)).$$ Note that the quotient is a set (or a groupoid equivalent to a set) because the action is free. In addition, it admits a natural fully faithful functor to $\Hom_P(Q, \Gamma(\oM_S))$, and it remains to prove that it is essentially surjective. In other words, we want to show that with respect to the flat topology, any homomorphism $Q\to\oM_S$ factors through $M_S$. This is done in \cite[Corollary~2.3]{Ostack}.
\end{proof}

\begin{remark}
In particular, morphisms from a log scheme to $\cX_P[Q]$ have no non-trivial automorphisms, so $\cX_P[Q]$ behaves as an algebraic space in the logarithmic category (see also \cite[Remark~6.4]{CCUW}). Technically, upgrading morphisms of schemes to morphisms of log schemes has this rigidifying effect because one gets rid of automorphisms of the log structures -- they do not correspond to automorphisms of the corresponding log schemes.
\end{remark}

\subsubsection{Globalization to fans}
Finally, we want to globalize the construction of $\cX_P[Q]$ to the case of a stacky fan $F$ over $P$. Intuitively this can be viewed as a combinatorial base change or pullback functor; in fact, in \cite[Definition~6.6]{CCUW}, the notation $a^*F$ is used. We prefer a more ad hoc approach that acts in two stages.

The contravariant functor $Q\mapsto\cX_P[Q]$ on the category of $P$-monoids can be also viewed as a covariant functor on the category of affine fans $F=\Spec(Q)$ over $P$, and we will use the same notation $F\mapsto\cX_P[F]$. This functor is compatible with fiber products (in fact, all finite limits) and takes open immersions to open immersions (face embeddings). Therefore, it globalizes and one obtains a functor $F\mapsto\cX_P[F]$ that associates to a fan $F$ over $P$ an $X$-stack $\cX_P[F]$ glued from open substacks $\cX_P[F_i]$, where $F=\cup_i F_i$ is an open affine cover of $F$. Clearly, this functor also respects finite limits.

\subsubsection{Globalization to stacky fans}
For any stacky fan $\cF$ choose a groupoid in fans $F_\bullet$ presenting it, and consider the simplicial $X$-stack $\cX_P[F_\bullet]$ whose $n$-th term is $\cX_P[F_n]$. Compatibility with fiber products implies that $\cX_P[F_\bullet]$ is a groupoid in a naive sense (rather than 2-categorically), because, being open immersions, all maps in $\cX_P[F_\bullet]$ have no automorphisms. The groupoid $\cX_P[F_\bullet]$ is thus \emph{inert} -- recall from \cite{ATWdestackification} that a map $f:X \rightarrow Y$ of stacks is called inert if the induced map $I_X \rightarrow f^*I_Y$ of inertia groups is an isomorphism, and a simplicial stack is inert if all maps appearing in its presentation are inert. By \cite[Lemma~2.1.4]{ATWdestackification} there exists a quotient stack $\cX=[\cX_P[F_\bullet]]$ of this groupoid, which is an honest algebraic stack provided with a morphism $\cX_P[F_0]\to\cX$ such that $\cX_P[F_{n+1}]=\cX_P[F_n]\times_\cX\cX_P[F_0]$. Moreover, any other presentation $F'_\bullet$ of $\cF$ is a groupoid equivalent to $F_\bullet$, hence $\cX_P[F_\bullet]$ and $\cX_P[F'_\bullet]$ are equivalent and have equivalent quotients $\cX$ and $\cX'$. This shows that $\cX_P[\cF]:=\cX$ defines an honest 2-functor from the 2-category of stacky fans over $P$ to the 2-category of stacks over $X$. It is compatible with fiber products.

\begin{example}\label{groupoidex}
(i) For $P$-monoids $Q_i$ set $\cX_P[Q_0\.Q_n]=\cX_P[J(Q_0\.Q_n)]$. It is the colimit of joint faces $\cX_P[F]$ in $\cX_P[Q_i]$.

(ii) We will use the notation $\cL_X=\cX_P[\cL_P]$, and we will later see that it depends only on the log scheme $X$. This stack is the quotient of the inert groupoid $\cX_\bullet$, where $\cX_n=\coprod_Q\cX_P[Q_0\.Q_n]$.

(iii) Assume that $X$ is a log point with $\oM_X=P$, say $\uX=\Spec(k)$ and the log structure is induced from $P\stackrel{0}\to k$. Then the natural maps $\cX_P[Q_0\.Q_n]\to J(Q_0\.Q_n)$ are homeomorphisms, and hence the underlying topological spaces $|\cL_X|$ and $|\cL_P|$ are homeomorphic. Equivalently, points of $|\cL_X|$ are parameterized by isomorphism classes of affine $P$-fans $\Spec(Q)$, and the generization relations correspond to sharpened localizations $Q\to Q'$ of fine $P$-monoids.
\end{example}





\section{Olsson's stack $\L X$}\label{section: logproperties}
In this section we construct a very concrete combinatorial presentations of Olsson's stack $\L X$. In particular, we will show that if $X$ possesses a global chart, then the natural morphism $\cL_X\to\L X$ is an equivalence. This will done without using any result about $\L X$, so as a by product we obtain a new proof of some properties of $\L X$.

\subsection{Recollections on the functor $\L{}$}\label{subsection: Olssonstack}
We start with reminding the reader of some of Olsson's definitions and results from \cite{Ostack}.

\subsubsection{The definition}
Let $X$ be a fine logarithmic scheme. In \cite{Ostack}, Olsson defined the stack $\L X$ parameterizing fine log schemes over $X$. It is the category fibered in groupoids over the category of $X$-schemes, whose objects are fine $X$-log schemes $Y \rightarrow X$ and morphisms are strict morphisms of $X$-log schemes. The fibration morphism simply forgets the log structures. In addition, the identity morphism of $X$ induces a tautological morphism $i_X\:X\to\L X$.

\subsubsection{Basic properties}
Olsson proved the following properties of the above construction.

\begin{theorem}\label{basicth}
(i) For any fine log scheme $X$, the stack $\L X$ is algebraic and of locally finite presentation over $X$, and its stabilizers are extensions of \'etale groups by diagonalizable ones.

(ii) The 2-functor $\L{}$ respects fibered products: if $T=Y\times_XZ$, then the natural morphism $\L T\to\L Y\times_{\L X}\L Z$ is an equivalence.

(iii) For any fine log scheme $X$, the morphism $i_X$ is an open immersion, which represents the substack of $\L X$ classifying $X$-log schemes $Y$ such that the morphism $Y\to X$ is strict. For any strict morphism of fine schemes $f\:Y\to X$, one has that $i_Y$ is the base change of $i_X$ with respect to $\L f$.
\end{theorem}

Let us say a couple of words about the proofs. Claim (ii) reduces to unwinding the definitions, see \cite[Proposition~3.20]{Ostack}, and we do not have anything new to say about it. Other claims will be reproved later on in this sections. Claim (iii) was proved by a relatively simple logarithmic technique in \cite[Proposition~3.19]{Ostack}. Part (i) is the most subtle and involved part of the theorem. Its first part is \cite[Theorem~1.1]{Ostack}, while the claim about stabilizers was not specially formulated, but follows easily from \cite[Corollary~5.25]{Ostack}.

\subsubsection{Covers by toric stacks}
Olsson also constructs an \'etale cover of $\L X$ by explicit relative toric stacks.

\begin{theorem}\label{coverth}
Assume $X$ is a fine log scheme, and let $I$ be the category of triples $(U,P,Q)$, where $U\to X$ is a strict \'etale morphism, $U\to Z_P$ is  a chart, and $P\to Q$ is a homomorphism of fine monoids. Then there exists a representable \'etale covering $\coprod_I\cU_P[Q]\to\L X$, and hence also a flat covering $\coprod_I U_P[Q]\to\L X$.
\end{theorem}

This is proved in \cite[Corollary~5.25]{Ostack}. We will refine this result by showing that the first cover induces an actual presentation of $\L X$ by an inert groupoid in stacks, and that this groupoid is nothing else but the groupoid $\cX_\bullet$ from Example~\ref{groupoidex}.

\subsection{A modular characterization of $\cX_\bullet$}
For the sake of comparison we will need logarithmic interpretations of the functors represented by the stacks $\cX_n$ and their building blocks.

\subsubsection{The case of $\cX_P[Q]$}
We start with the simplest case. We have already described morphisms of log schemes to $\cX_P[Q]$, but studying morphisms of schemes is a finer question since non-trivial automorphisms show up. Given an $\uX$-scheme $f\:\uS\to\uX$ with a log structure $M_S$, by an {\em $(X,Q)$-structure} on $M_S$ we mean homomorphisms $f^*M_X\to M_S$ and $\gamma\:Q\to\oM_S$ such that the compositions $P\to f^*M_X\to\oM_S$ and $P\to Q\to\oM_S$ coincide and flat-locally on $\u S$ one can lift $\gamma$ to a chart $Q\to M_S$. By $\cX'_P[Q]$ we denote the category fibred over the category of $\u X$-schemes, whose fiber over $\u S$ is the groupoid of log structures $M_S$ on $S$ provided with an $(X,Q)$-structure.

\begin{remark}
Thus, $S=(\uS,M_S)$ is nothing but an $X$-log scheme such that flat-locally on $S$ the morphism $S\to X$ possesses a chart $S\to X_P[Q]$. Non-uniqueness in the choice of a chart corresponds to non-triviality of automorphism groups in $\cX'_P[Q]$.
\end{remark}

Now, we extend \cite[Proposition 5.15]{Ostack} to the relative case.

\begin{theorem}\label{toricstackth}
Let $P\to Q$ be a homomorphism of fine monoids and let $X$ be a log scheme with a global chart $X\to Z_P$. Then the stacks $\cX'_P[Q]$ and $\cX_P[Q]$ are naturally equivalent.
\end{theorem}
\begin{proof}
Set $\cX=\cX_P[Q]$ and $\cX'=\cX'_P[Q]$. Given an object $S\to X$, $Q\to\oM_S$ of $\cX'$, let us construct a canonical morphism $S\to\cX$. By flat descent, one can work flat-locally on $S$, hence we can assume that there exists a chart $S\to X_P[Q]$. Two charts $a\:Q\to M_S$ and $b\:Q\to M_S$ differ by a homomorphism $a-b\:Q\to\cO^\times_S$, hence they induce the same morphism $S\to\cX_P[Q]$. This construction is functorial in $M_S$, so we obtain a functor $\phi\:\cX'\to\cX$.

Conversely, there is an obvious $(X,Q)$-structure on $X_0=X_P[Q]$ and two its pullbacks to $X_1=X_0\times_\cX X_0$ coincide. (It is easy to see that the induced homomorphisms $Q\to M_{X_1}$ differ by a homomorphism $Q\to\cO^\times_{X_1}$, hence their sharpenings coincide.) Therefore, this $(X,Q)$-structure descends to $\cX$, and for any morphism of schemes $\uS\to\cX$, the pulled back log structure acquires a canonical $(X,Q)$-structure pulled back from that on $\cX$. Clearly, this construction is also functorial and we obtain a functor $\psi\:\cX\to\cX'$.

Checking that $\phi$ and $\psi$ are essentially inverse is straightforward, and we skip the details.
\end{proof}

\begin{corollary}\label{reltoriccor}
Let $f\:\uS\to\cX_P[Q]$ be a morphism of $X$-schemes, and $M_S$ the induced log structure on $\uS$ with an $(X,Q)$-structure. Then $\Aut(f)$ is naturally isomorphic to the group of $M_X$-automorphisms of $M_S$ that induce identity on $\oM_S$.
\end{corollary}
\begin{proof}
By Theorem \ref{toricstackth}, $\Aut(f)$ is the group of $M_X$-automorphisms $\sigma\:M_S\to M_S$ that are compatible with the homomorphism $\overline{\alpha}\:Q\to\oM_S$. Since the latter locally lifts to a chart, $\overline{\alpha}$ is surjective (when viewing $Q$ as a constant sheaf). In particular, compatibility with $\overline{\alpha}$ takes place if and only if $\overline{\sigma}\:\oM_S\to\oM_S$ is the identity.
\end{proof}

\begin{remark}
In fact, $\Hom_X(S,\cX)$ in the category of log stacks is the set of isomorphism classes of objects in the groupoid $\HOM_\uX(\uS,\cX)$. In particular, as in \cite{Ostack}, this could be used to deduce Lemma~\ref{toriclem} from Theorem~\ref{toricstackth}.
\end{remark}

\subsubsection{The case of $\cX_P[Q_0\.Q_n]$}
Now, let us extend the above results to a tuple $Q=(Q_0\.Q_n)$ of $P$-monoids. By an $(X,Q)$ structure on $M_S$ we mean a tuple of $(X,Q_i)$-structures that share the same homomorphism $f^*M_X\to M_S$. By $\cX'_P[Q]$ we denote the fibered category over the category of $\uX$-schemes whose fiber over $\uS$ is the groupoid of log structures on $\uS$ with an $(X,Q)$-structure. As the reader might guess, this generalization was designed to extend Theorem~\ref{toricstackth} as follows:

\begin{theorem}\label{toricstackth1}
Let $P\to Q_i$, $i=0\. n$ be homomorphisms of fine monoids and let $X$ be a log scheme with a global chart $X\to Z_P$. Then the stacks $\cX'_P[Q_0\.Q_n]$ and $\cX_P[Q_0\.Q_n]$ are naturally equivalent.
\end{theorem}
\begin{proof}
Set $Q=(Q_0\.Q_n)$ for brevity. The pullbacks to $\cX_P[Q]$ of the $(X,Q_i)$-structures on $\cX_P[Q_i]$ form a canonical $(X,Q)$-structure on $\cX_P[Q]$, hence we obtain a morphism $\cX_P[Q]\to\cX'_P[Q]$.

Conversely, assume that $f\:S\to X$ is a morphism of log schemes, and $f^*M_X\to M_S$ is upgraded to an $(X,Q)$-structure by homomorphisms $\overline{\alpha_i}\:Q_i\to\oM_S$. We will assign to this datum a canonical morphism $\uS\to\cX_P[Q]$. This can be done flat-locally on $\uS$, so we can assume that $S$ has a global chart $N=\Gamma(\oM_S)\to M_S$ and, in addition, each homomorphism $Q_i\to\oM_S$ lifts to a chart $Q_i\to M_S$. In this case, the induced homomorphisms $Q_i\to N$ are sharpened localizations, hence $\Spec(N)$ is a joint face of $\Spec(Q_i)$, $0\le i\le n$, and we obtain a morphism $h_{\overline{\alpha}}\:S\to X_P[N]\into X_P[Q]$. This provides a functor $\cX'_P[Q]\to\cX_P[Q]$.

Again, the check that the two functors are essentially inverse is straightforward. For example, it is clear from the construction that $h_{\overline{\alpha}}$ induces the $(X,Q)$-structure corresponding to $Q_i\to N\to M_S\to\oM_S$, and hence $\cX'_P[Q]\to\cX_P[Q]\to\cX'_P[Q]$ is equivalent to the identity.
\end{proof}

\subsection{The presentation of $\L X$}

\subsubsection{The local case}
We start with the case when $X$ possesses a global chart, and let the toric $X$-stack $\cL_X=\cX_P[\cL_P]$ and its groupoid presentation $\cX_\bullet$ be as in Example~\ref{groupoidex}(ii). Since $\cX_n$ is a groupoid of $X$-log schemes, we obtain compatible morphisms $\cX_n\to\L X$ and hence also $\cL_X\to\L X$.

\begin{theorem}\label{logXth}
If $X$ is a log scheme $X$ possessing a global chart $X\to Z_P$, then the morphism $\cL_X\to\L X$ is an equivalence. In particular, $\cL_X$ depends only on $X$, but not on the choice of the chart.
\end{theorem}
\begin{proof}
If $\uS$ is an $X$-scheme with an $X$-morphism $\uS\to\L X$, then there exists a flat covering $\coprod_i\uS_i\to\uS$ such that the induced log structure on each $\uS_i$ possesses a global chart $\cX_P[Q_i]$. Then $\uS_i\to\L X$ factors through $\cX_P[Q_i]$, and hence $\coprod_i\uS_i\to\L X$ factors through $\cX_0=\coprod_Q\cX_P[Q]$. This implies that the morphism $\cX_0\to\L X$ is essentially surjective, and it remains to show that the induced groupoid $\cX'_\bullet$ given by $\cX'_0=\cX_0$ and $\cX'_{n+1}=\cX'_n\times_{\L X}\cX'_0$ is equivalent to $\cX_\bullet$.

Clearly, $\cX'_n$ is the disjoint union of fiber products over $\L X$ of tuples $\cX'_P[Q_0]\.\cX'_P[Q_n]$. By definition, such a product represents the functor associating to $\uS$ the following dataum: log structures $M_i$ on $\uS$ with $(X,Q_i)$-structures for $0\le i\le n$ and isomorphisms $M_i\toisom M_{i+1}$ for $0\le i <n$. Clearly, this is nothing else but an $(X,Q_0\.Q_n)$-structure, so such a fiber product is tautologically equivalent to $\cX'_P[Q_0\.Q_n]$. By Theorem~\ref{toricstackth1} it is equivalent to $\cX_P[Q_0\.Q_n]$ and taking disjoint union over all tuples of $P$-monoids we obtain that $\cX'_n\toisom\cX_n$. Naturally, this is compatible with the maps in $\cX'_\bullet$ and $\cX_\bullet$, and we skip this check.
\end{proof}

\subsubsection{The general case}
Theorem \ref{logXth} generalizes straightforwardly to the case when $X$ possesses a global fan, which does not have to be affine. However, not every log scheme possesses a global fan, so we prefer to globalize this theorem by a more ad hoc tool.

\begin{lemma}\label{descentlem}
Assume that $X_0\to X$ is a strict flat morphism of log schemes and consider the associated groupoid $X_{n+1}=X_n\times_XX_0$. Then $\L {X_\bullet}$ is an inert groupoid in stacks and its quotient is equivalent to $\L X$.
\end{lemma}
\begin{proof}
For any morphism of log schemes $S\to X_0$, an $M_X$-automorphism of $M_S$ is the same as an $M_{X_0}$-automorphism. So, if $f_0\:\uS\to\L{X_0}$ and $f\:\uS\to\L X$ are the corresponding morphisms, then $\Aut(f)=\Aut(f_0)$. This proves that the map $\L{X_0}\to\L X$ is inert. Furthermore, the maps in $X_\bullet$ are strict, hence the maps in $\L{X_\bullet}$ are inert. Since the 2-functor $\L {}$ commutes with fiber products, we obtain all remaining claims.
\end{proof}

\subsection{Applications}
Now, we can apply the theory developed above to provide new proofs of Olsson's theorems.

\subsubsection{Proof of Theorem~\ref{coverth}}
If $X$ possesses a global chart, then this is an obvious corollary of Theorem~\ref{logXth} and it suffices to take $U=X$. In general, we fix a strict \'etale covering $\coprod_iU_i\to X$ such that each $U_i$ possesses a global chart and use that $\coprod\L{U_i}\to\L X$ is a covering, for example, by Lemma~\ref{descentlem}.

\subsubsection{Proof of Theorem~\ref{basicth}(i)}
If $X$ possesses a global chart $Z_P$, then by Theorem~\ref{logXth} $\L X$ is the quotient of the inert groupoid $\cX_\bullet$ in algebraic stacks of locally finite presentation over $X$. By \cite[Lemma~1.4]{ATWdestackification} this quotient is an algebraic stack of locally finite presentation over $X$. Moreover, the maps $\cX_1\rightrightarrows\cX_0$ are \'etale (even disjoint unions of open immersions), hence the stabilizers of $\L X$ are \'etale extensions of the stabilizers of $\cX_0$, and the latter are diagonalizable groups.

If $X$ is arbitrary, then we fix a strict \'etale covering $X_0\to X$ by a disjoint union of log schemes admitting a global chart. Clearly, each $X_n$ in the associated groupoid is also a disjoint union of log schemes admitting a global chart. By Lemma~\ref{descentlem} $\L{X_\bullet}$ is an inert groupoid, and its components are algebraizable of locally finite presentation by the local case. Using \cite[Lemma~1.4]{ATWdestackification} once again, we obtain that $\L X$ is as required.

\subsubsection{Proof of Theorem~\ref{basicth}(iii)}
Using \'etale descent, this claim also easily reduces to the case when $X$ possesses a global chart $X\to Z_P$. Then the morphism $i_X\:X\to\L X$ is induced by the trivial face $\cX_P[P]=X$ of the cover $\coprod_Q\cX_P[Q]\to\L X$, and hence $i_X$ is \'etale by Theorem~\ref{coverth}. Furthermore, $X\times_{\L X}X=\cX_P[P,P]$ by Theorem~\ref{toricstackth1}, and $\cX_P[P,P]=X$ in a tautological way. Thus, $i_X$ is a monomorphism and hence an open immersion.

\begin{section}{Comparison of Log and Weak Log Properties}

\subsection{The definitions}
Let $f: X \rightarrow Y$ be a morphism of log schemes. As before, we denote by $\L f$ the associated morphism $\L X \rightarrow \L Y$; we write $w\L f$ for the composition $\L f \circ i_X: X \rightarrow \L X \rightarrow \L Y$. In \cite[Definition 4.1]{Ostack}, Olsson gives the following definition:

\begin{definition}
Let $\c{P}$ be a property of representable morphisms of algebraic stacks. We say that a morphism $f: X \rightarrow Y$ of fine log schemes has property log $\c{P}$ (resp. weak log $\c{P})$ if $\L f$ (resp. $w \L f$) has property $\c{P}$.
\end{definition}

Properties log $\c{P}$ and weak log $\c{P}$ are not equivalent in general. It is however convenient to know the equivalence of log $\c{P}$ and weak log $\c{P}$ when possible, as the former has better functorial properties while the latter is simpler to work with. Our next goal is to give a simple criterion which ensures equivalence of log $\c{P}$ and weak log $\c{P}$. \\

\subsection{The comparison theorem}
The key result that allows us to compare properties of $\L f$ and $w \L f$ is the following proposition, which shows that \'etale charts of $\L f$ are obtained by pulling back \'etale charts of $w\L f$.

\begin{proposition}
\label{prop:Cartesian}
Let $f: X \rightarrow Y$ be a morphism of log schemes, and assume that $f$ admits a global injective chart $P \rightarrow Q$. Let $Q \rightarrow R$ be an injective homomorphism of monoids. Then there is a 2-cartesian diagram
\begin{align*}
\xymatrix{\c{X}_Q[R] \ar[d] \ar[r] & X\ar[d]\\
\c{Y}_P[R]\ar[r]& \c{Y}_P[Q].}
\end{align*}
\end{proposition}
\begin{proof}
In fact, the claim is that
$$X\times_{\cZ_Q}\cZ_R\toisom X\times_{Y\times_{\cZ_P}\cZ_Q}Y\times_{\cZ_P}\cZ_R.$$
This is a general property of 2-fiber products which is straightforward to check.
\end{proof}

\begin{theorem}
\label{theorem: logPiffwlogP}
Let $\c{P}$ be a property of representable morphisms of algebraic stacks which is \'etale local on the source and target, is stable under composing with \'etale maps, and is additionally stable under pullbacks with respect to morphisms locally of finite type.\footnote{Properties of morphisms of schemes which are \'etale local on the source, target, and stable under postcomposing with \'etale maps are called \emph{\'etale local on the source-and-target} in \cite[\href{https://stacks.math.columbia.edu/tag/04QW}{Tag 04QW}]{stacks-project}. It is shown in loc. cit. that their stability under precomposing with \'etale maps and with respect to \emph{\'etale} pullbacks is in fact automatic.} Then a morphism $f:X \rightarrow Y$ of log schemes has property log $\c{P}$ if and only if it has property weak log $\c{P}$.
\end{theorem}
\begin{proof}
As $\c{P}$ is stable under composing with \'etale maps, and $X \rightarrow \L X$ is an open immersion, property $\c{P}$ for $\L f$ implies property $\c{P}$ for $w \L f$. So log $\c{P}$ implies weak log $\c{P}$. Assume conversely that $w \L f: X \rightarrow \L Y$ has property $\c{P}$. As $\c{P}$ is is \'etale local on $X$, we can assume that $f$ admits a global chart $P \rightarrow Q$. Consider the following commutative diagram, where $Q \rightarrow R$ is an arbitrary injective map of monoids:

\[
\begin{tikzcd}
\L X \ar[d, "\L f"] & \c{X}_Q[R] \ar[l, "a"] \ar[r] \ar[d] & X \ar[r,"\textup{id}"] \ar[d,"g"] & X \ar[d, "w \L f"]\\ \L Y & \c{Y}_P[R] \ar[l,"b"] \ar[r]  & \c{Y}_P[Q] \ar[r,"c"] & \L Y
\end{tikzcd}
\]

\noindent Here the maps $a,b,c$ are the maps of the \'etale representable cover of Theorem \ref{coverth}. As $w \L f = c \circ g$ has $\c{P}$ and $c:\c{Y}_P[Q] \rightarrow \L Y$ is \'etale, the map $g:X \rightarrow \c{Y}_P[Q]$ must also have property $\c{P}$: pulling back $c \circ g: X \to \L Y$ along $c$, we see that the projection map $p_2:X \times_{\L Y} \c{Y}_P[Q] \rightarrow \c{Y}_P[Q]$ has $\c{P}$, and the projection $p_1:X \times_{\L Y} \c{Y}_P[Q] \rightarrow X$ is \'etale; the section $(\textup{id},g): X \rightarrow X \times_{\L Y} \c{Y}_P[Q]$ is thus \'etale, and $g = p_2 \circ (\textup{id},g)$ has $\c{P}$, as $\c{P}$ is stable under composing with \'etale maps. Next, since the morphism $\c{Y}_P[R] \rightarrow \c{Y}_P[Q]$ is locally of finite type, Proposition \ref{prop:Cartesian} implies that $\c{X}_Q[R] \rightarrow \c{Y}_P[R]$ has property $\mathcal{P}$. Composing with the \'etale map $b$, we see that $\c{X}_Q[R] \rightarrow \L Y$ has $\c{P}$. But the stacks $\c{X}_Q[R]$ cover $\L X$ as $Q \rightarrow R$ ranges through all injective maps, and $\c{P}$ is \'etale local on the source, so $\L f$ has property $\mathcal{P}$. Thus weak log $\c{P}$ implies log $\c{P}$.
\end{proof}

\end{section}

\begin{section}{Log Regularity}
\label{section: logreg}

\subsection{The definition}
The main new application of Theorem \ref{theorem: logPiffwlogP} concerns logarithmic regularity. Recall that classically a morphism of schemes $\u{f}:\u{X} \rightarrow \u{Y}$ is called regular if it is flat and its fibers are locally Noetherian and geometrically regular. As regularity is local for the smooth topology on both source and target, the definition extends to morphisms of algebraic stacks in a straightforward manner.

\begin{definition}
\label{defn: log-reg}
A morphism $f: X \rightarrow Y$ is log regular if $\L f: \L X \rightarrow \L Y$ is regular. It is weakly log regular if $w \L f: \u X \rightarrow \L Y$ is regular.
\end{definition}


\begin{lemma}
\label{cor: properties}
Let $f\:X\to Y$, $g\:Y\to Z$ and $\phi\:Y'\to Y$ be morphisms of log schemes with the base change $f'\:X'=X\times_YY'\to Y'$.
\begin{enumerate}
\item Stability under compositions: if $f$ and $g$ are log regular, then $g\circ f$ is log regular.
\item Stability under base changes: if $f$ is log regular and $f'$ has locally Noetherian fibers (the latter is automatic if $X'$ is locally Noetherian), then $f'$ is log regular.
\item Log flat descent: if $f'$ is log regular and $\phi$ is log faithfully flat, then $f$ is log regular. In particular, log regularity is log fppf local on the base.
\item If $g \circ f$ is log regular and $f$ is a log fpqc cover, then $g$ is log regular. In particular, log regularity is log smooth local on the source.
\item If $g \circ f$ is log regular and $g$ is log \'etale, then $f$ is log regular.
\item If $f$ is strict, it is log regular if and only if it is regular.
\end{enumerate}
\end{lemma}
\begin{proof}
We claim that if the morphism $f'$ has locally Noetherian fibers, then the fibers of $\L{f'}\:\L{X'}\to\L{Y'}$ are also locally Noetherian. Since $\L{}$ is compatible with base changes, it suffices to prove this when $Y'$ is a point. Then $X'$ is locally Noetherian, and since $\L{X'}$ is locally of finite presentation over $X'$, $\L{X'}$ and the fibers of $\L{f'}$ are locally Noetherian. Now, claims (1)--(5) follow immediately by replacing $X \rightarrow Y$ with $\L X \rightarrow \L Y$ and invoking the well known classical analogues. The assumption on $f'$ in (2) is needed because in general local Noetherianity can be lost under base field extensions, and descent of local Noetherianity requires the quasi-compactness assumption in (4).

Finally, if $f$ is strict, then the map $X \rightarrow \L Y$ factors through $Y$, and since the map $Y \rightarrow \L Y$ is \'etale and $X \rightarrow \L Y$ is regular, we can conclude.
\end{proof}

As regularity is stable under pullbacks and can be checked \'etale locally on the source and flat locally on the base, Theorem \ref{theorem: logPiffwlogP} yields

\begin{theorem}
\label{theorem: equivalence}
A map $f:X \rightarrow Y$ is log regular if and only if it is weakly log regular.
\end{theorem}

Finally, since the canonical maps between $X$ and its integralization and its saturation are log \'etale, we obtain

\begin{corollary}
\label{cor: saturation}
Let $?$ and $*$ denote the identity, integralization, or saturation functor, and assume that $f\:X \rightarrow Y$ is a morphism of log schemes such that the morphism $f'\:X^{?} \rightarrow Y^{*}$ is defined. Then $f$ is log regular if and only if $f'$ is log regular.
\end{corollary}

\subsection{Chart Criterion}
\label{section: Chart Criterion} The exposition in this section follows Ogus' treatment of logarithmic smoothness in \cite{oguslog}. To extend the arguments to the log regular case, we need to establish a few results about regular morphisms of stacks. There are two possible approaches to this. One is to use Olsson's theory of the log cotangent complex \cite{Ocot}. The other, which we take, is elementary -- the only technical tool required is smooth descent -- but it is longer as we need to develop some results about differentials of stacks, which are probably well known. Both approaches rely on the next lemma and the following corollary \footnote{This is the reason we need to impose some Noetherian hypotheses.}:

\begin{lemma}
\label{injectiveimpliesreg}
Let
\begin{align*}
\xymatrix{ X \ar[r]^f & Y \ar[r]^g & Z}
\end{align*}
be morphisms of locally Noetherian schemes, and suppose that $g$ is smooth and $g \circ f$ is regular at $x \in X$. Then $f$ is regular at $x$ if and only if the map $f^*\Omega_{Y/Z} \otimes k(x) \rightarrow \Omega_{X/Z} \otimes k(x)$ is injective.
\end{lemma}

\begin{proof}
The statement is local, so we may assume that $X,Y,Z$ are the spectra of the local Noetherian rings $A,B,C$ respectively, with $x$ corresponding to the maximal ideal $m$ of $A$. By a theorem of Andr\'{e} -- see for instance \cite[Property 2.8]{Spiv} --,  regularity of $f$ is equivalent to showing that the first Andr\'{e}-Quillen homology group $H_1(X/Y;k(\mathfrak{p}))$ vanishes for every prime ideal $\mathfrak{p}$ in $A$. Recall that this homology group is the first homology group of the complex $L_{X/Y} \otimes k(\mathfrak{p})$, where $L_{X/Y}$ is the cotangent complex of $f$. The exact triangle
\begin{align*}
\xymatrix{L_{X/Y}[-1] \ar[r] & f^*L_{Y/Z} \ar[r] & L_{X/Z} \ar[r] & L_{X/Y} \ar[r] & }
\end{align*}
yields a long exact sequence

\begin{align*}
\xymatrix{0 \ar[r] & H_1(X/Y,k(\mathfrak{p})) \ar[r] & f^*\Omega_{Y/Z} \otimes k(\mathfrak{p}) \ar[r] & \Omega_{X/Z} \otimes k(\mathfrak{p}) \ar[r] & \Omega_{X/Y} \otimes k(\mathfrak{p}) \ar[r] & 0}
\end{align*}

\noindent where the $0$ on the left follows from the regularity of $g \circ f$. Thus, if $f$ is regular at $x$, the map $f^*\Omega_{Y/Z} \otimes k(x) \rightarrow \Omega_{X/Z} \otimes k(x)$ is clearly injective, and, conversely, regularity of $f$ follows if $f^*\Omega_{Y/Z} \otimes k(\mathfrak{p}) \rightarrow \Omega_{X/Z} \otimes k(\mathfrak{p})$ is injective for all prime ideals $\mathfrak{p} \subset A$. The regularity of $g \circ f$ further implies that $\Omega_{X/Z}$ is flat, and the smoothness of $g$ that $f^*\Omega_{X/Z}$ is a finitely generated free module of some rank $n$. By Lazard's theorem, we may write $\Omega_{X/Z}$ as an increasing limit $\dirlim F_i$ of finitely generated free modules, each containing the image of $f^*\Omega_{Y/Z}$. Assuming then that $f^*\Omega_{Y/Z} \otimes k(x) \rightarrow \Omega_{X/Z} \otimes k(x)$ is injective at the closed point, it follows that each map $f^*\Omega_{Y/Z} \otimes k(x) \rightarrow F_i \otimes k(x)$ is an injective map of finitely generated free modules; it is thus given by a matrix over $k(x) = A/m$ that has a non-vanishing $n \times n$ minor. This minor is thus a unit, and remains non-vanishing over each $k(\mathfrak{p})$ for any prime ideal $\mathfrak{p} \subset A$, which shows that $f^*\Omega_{Y/Z} \otimes k(\mathfrak{p}) \rightarrow F_i \otimes k(\mathfrak{p})$ remains injective. As taking a filtered direct limit of $A$-modules preserves exactness, it follows that $f^*\Omega_{Y/Z} \otimes k(\mathfrak{p}) \rightarrow \Omega_{X/Z} \otimes k(\mathfrak{p})$ remains injective, and that therefore $f$ is regular.
\end{proof}

\begin{corollary}
\label{injectiveimpliesregstack}
Let
\begin{align*}
\xymatrix{ X \ar[r]^f & Y \ar[r]^g & Z}
\end{align*}
be representable morphisms of locally Noetherian stacks, with $g$ smooth and $g \circ f$ regular at $x \in X$. Then $f$ is regular at $x$ if and only if the natural map $f^*\Omega_{Y/Z} \otimes k(x) \rightarrow \Omega_{X/Z} \otimes k(x)$ is injective.
\end{corollary}

\begin{proof}
Let $Z' \rightarrow Z$ be a smooth cover, and form the diagram
\begin{align*}
\xymatrix{X' \ar[r] \ar[d] & Y' \ar[r] \ar[d] & Z' \ar[d] \\ X \ar[r] & Y \ar[r] & Z}
\end{align*}
where each square is Cartesian. By the representability hypothesis, $X'$ and $Y'$ are also schemes. Since for a Cartesian diagram
\begin{align*}
\xymatrix{T' \ar[r]^{l} \ar[d] & T \ar[d] \\ S' \ar[r]_k & S}
\end{align*}
with $k$ smooth, we have $l^*L_{T/S} = L_{T'/S'}$, the result follows from Lemma \ref{injectiveimpliesreg}.
\end{proof}

Let now $X \rightarrow Y \rightarrow Z$ be morphisms of log schemes. To apply the preceeding discussion to log regularity, we must understand the pullback of $\Omega_{\L Y/ \L Z}$ to $X$. We fix some notation: to each $\c{O}_X$ module $I$, we associate the scheme $X[I]$ defined by the sheaf of algebras $\c{O}_X[I] = \c{O}_X \oplus I$ over $\c{O}_X$, with multiplication of any two elements of $I$ defined to be $0$. The projection $p:X[I] \rightarrow X$ comes with an evident section $s: X \rightarrow X[I]$ corresponding to the projection of algebras $\c{O}_X[I] \rightarrow \c{O}_X$.
\begin{lemma}
\label{lifts}
Let $g:Y \rightarrow Z$ be a representable morphism of stacks, and $f:X \rightarrow Y$ a morphism from a scheme. Then $f^*\Omega_{Y/Z}$ co-represents the functor that takes an $\c{O}_X$-module $I$ to the set $\Hom^X_{Z}(X[I],Y)$ of isomorphism classes of 2-commutative lifts
\[
\begin{tikzcd}
X \ar[r,"f"] \ar[d, "s"] & Y \ar[d, "g"] \\ X[I] \ar[r,"g \circ f \circ p"] \ar[ur, dotted] & Z
\end{tikzcd}
\]
\end{lemma}

\begin{proof}
In the case when $Y$ and $Z$ are schemes this is immediate, as there is a canonical isomorphism between such lifts and derivations from $\c{O}_Y$ to $I$ over $\c{O}_Z$. In general, we can reduce to the case of schemes by observing that we have a bijection $\Hom_Z^{X}(X[I],Y) = \Hom_{X[I]}^X(X[I],X[I] \times_Z Y)$, obtained by using the representability of $g$ and factoring a given lift as
\[
\begin{tikzcd}
X \ar[r,"f \times s"] \ar[d, "s"] & Y \times_Z X[I] \ar[r,"\textup{pr}_1"] \ar[d] & Y \ar[d, "g"] \\ X[I] \ar[r,"\textup{id}"] \ar[ur, dotted] & X[I] \ar[r,"g \circ f \circ p"] & Z
\end{tikzcd}
\]
The statement about schemes then gives a natural isomorphism $\Hom_{Z}^X(X[I],Y) = \Hom_{X[I]}^X(X[I],Y \times_Z X[I]) \cong \Hom((f \times s)^*\Omega_{Y \times_Z X[I]/X[I]},I) \cong \Hom_{\c{O}_X}(f^*\Omega_{X/Y}, I)$.

\end{proof}

\begin{corollary}
\label{pullbackdifferentials}
Let
\[
\xymatrix{ X \ar[r]^f & Y \ar[r]^g & Z}
\]
be morphisms of fine log schemes. Then $w\L f^* \Omega_{\L Y/ \L Z} = f^*\Omega_{Y/Z}^{\log}$.
\end{corollary}

\begin{proof}
By Lemma \ref{lifts}, the pullback $w\L f^* \Omega_{\L Y/ \L Z}$ co-represents the functor that assigns to the quasicoherent sheaf $I$ the set of lifts
\[
\begin{tikzcd}
X \ar[r,"w\L f"] \ar[d, "s"] & \L Y \ar[d, "\L g"] \\ X[I] \ar[r] \ar[ur, dotted] & \L Z
\end{tikzcd}
\]
where the bottom map is $\L g \circ w \L f \circ p$. By definition, such diagrams correspond to diagrams of log schemes
\[
\begin{tikzcd}
X \ar[r,"f"] \ar[d, "s"] & Y \ar[d, "g"] \\ X[I] \ar[r,"g \circ f \circ p"] \ar[ur, dotted] & Z
\end{tikzcd}
\]

\noindent where, as the morphism $X \rightarrow \L Z$ factors through $X[I]$, $X[I]$ has been given a log structure that makes $X \rightarrow X[I]$ strict. In other words, the log structure $M_{X[I]} = M_{X} \oplus_{\c{O}_X^*} \c{O}_{X[I]}^* = M_{X}[1+I]$, where $I$ has been embedded into the units of $\c{O}_{X[I]}$ by $x \mapsto 1 + x$. Such diagrams correspond precisely to a log derivation of $(\c{O}_Y,M_Y)$ into $I$ over $(\c{O}_Z,M_Z)$, and thus are in natural bijection with $\Hom_{\c{O}_X}(f^*\Omega_{Y/Z}^{\log},I)$. Therefore $w\L f^* \Omega_{\L Y/\L Z} = f^* \Omega_{Y/Z}^{\log}$.
\end{proof}

Suppose that $X$ is a log scheme, and $m_1,\cdots,m_n$ are elements of $\Gamma(X,M_X)$. From \ref{toriclem}, the map $\NN^n \rightarrow \Gamma(X,M_X)$, which sends the $i$-th generator $e_i$ to $m_i$ induces a map $X \rightarrow Z_{\NN^n} = \mathbb{A}^n$.

\begin{lemma}
\label{lem:regsequence}
Suppose $f: X \rightarrow Y$ is a log regular map of locally Noetherian schemes, and let $m_1,\cdots,m_n \in M_X(X)$ be elements such that their images $d\log m_i \in \Omega^{\log}_{X/Y,x} \otimes k(x)$ are linearly independent at $x \in X$. Then the induced map $g: X \rightarrow Y':= Y \times \mathbb{A}^n$ is log regular at $x$.
\end{lemma}

\begin{proof}
Consider the induced sequence of morphisms $X \rightarrow \L {Y'} \rightarrow \L Y$. By Corollary \ref{injectiveimpliesregstack}, it suffices to check that the map $\Omega_{\L {Y'}/\L Y} \otimes k(x) \rightarrow \Omega_{X/\L Y} \otimes k(x)$ is injective. By Corollary \ref{pullbackdifferentials}, this map is isomorphic to the map of log differentials $f^*\Omega_{Y'/Y}^{\log} \otimes k(x) \rightarrow \Omega_{X/Y}^{\log} \otimes k(x)$. As sheaf $\Omega_{Y'/Y}$ is locally free, generated by the elements $d \log e_i$, which map to $d \log m_i \in \Omega_{X/Y}^{\log}$, injectivity follows by the hypothesis.
\end{proof}

As a corollary, we obtain the following technical lemma, which will be used in the proof of \ref{theorem: Chart Criterion}. From here on, for an \'etale sheaf $\c{F}$, we use the notation $\c{F}_{\o{x}}$ for the stalk of $\c{F}$ at $x$ in the \'etale topology, and $k(\o{x})$ for the residue field of $\c{O}_{X,\o{x}}$, i.e. a separable closure of $k(x)$.

\begin{lemma}
\label{lem: factorization}
Suppose $f:X \rightarrow Y$ is a log regular map between locally Noetherian fine log schemes. Then, \'etale locally around each point $x \in X$, there is a factorization of $X \rightarrow Y$ to $X \rightarrow Y':= Y \times \mathbb{A}^n \rightarrow Y$ such that $X \rightarrow Y'$ is log regular and $\overline{M}_{X/Y',\o{x}}^{\rm{gp}} \otimes k(\o{x})=0$.
\end{lemma}

\begin{proof}
Consider the exact sequence of log differentials $\Omega_{X/Y} \otimes  k(\o{x}) \rightarrow \Omega_{X/Y}^{\rm{log}} \otimes k(\o{x}) \rightarrow \overline{M}_{X/Y,\o{x}}^{\rm{gp}} \otimes k(\o{x})$ and take elements $m_1,\cdots,m_n \in M_{X,\o{x}}$ such that the image of $d\log m_i$ forms a basis for $\overline{M}_{X/Y,\o{x}}^{\rm{gp}} \otimes k(\o{x})$. Take some \'etale neighborhood $U$ of $x$ on which the elements $m_i$ are defined, and a lift $y$ of $x$. Then the elements $d\log m_i \in \Omega_{U/Y,y} \otimes k(y)$ are linearly independent, by assumption, so the map $U \rightarrow Y'$ is log regular at $y$ from  \ref{lem:regsequence}. Since $\o{M}_{X/Y',\o{x}}^{\gp} \otimes k(\o{x}) = \o{M}_{U/Y',\o{y}}^{\gp} \otimes k(\o{y})$, the conclusion follows.
\end{proof}

Next, fix a base ring $A$. As in subsection \ref{subsection: Olssonstack}, we denote $A_P = \Spec A[P]$ and $\mathcal{A}_P = [\Spec A[P]/ \Spec A[P^{\rm{gp}}]]$.

\begin{lemma}
\label{lem:log-smooth}
Let $P \rightarrow Q$ be an injective homomorphism of monoids. Then the map $A_Q \rightarrow \mathcal{A}_Q \times_{\mathcal{A}_P} A_P$ is smooth if and only if the torsion part of the cokernel of $P^{\rm{gp}} \rightarrow Q^{\rm{gp}}$ has order invertible in $A$.
\end{lemma}

\begin{proof}
Since the map $\mathcal{A}_Q \times_{\mathcal{A}_P}A_P \rightarrow \L {A_P}$ is \'etale, the map $A_Q \rightarrow \mathcal{A}_Q \times_{\mathcal{A}_P} A_P$ is smooth if and only if the map $A_Q \rightarrow \L {A_P}$ is smooth. This is equivalent to log smoothness of $A_Q \rightarrow A_P$, which is equivalent to the statement about the cokernel -- see for instance \cite{Klog}.
\end{proof}

Before proceeding to a chart criterion for log regularity, we recall for the convenience of the reader the notion of a neat chart. Given a morphism $f: X \rightarrow Y$ of log schemes, a chart $\phi: P \rightarrow Q$ for $f$ at a point $x \in X$ is called neat at $x$ if $\phi$ is injective and the induced map $Q^\gp/P^\gp \rightarrow \o{M}^\gp_{X/Y,\o{x}}$ is an isomorphism. In order to avoid awkward notation, we use the following convention when discussing charts defined in different topologies. Let $y = f(x) \in Y$. Given a chart $P \rightarrow M_Y$ defined fppf locally around $y$, we say that a chart for $f$ exists \'etale (resp. fppf) locally around $x \in X$ if the following holds: given an fppf neighborhood $p: Y' \rightarrow Y$ of $y$ on which the chart $P \rightarrow p^*M_Y = M_{Y'}$ is defined, there exists an \'etale (resp. fppf) neighborhood $X'' \rightarrow X \times_Y Y' = X'$ of a lift $x' \in X'$ of $x \in X$ and a chart $Q \rightarrow M_{X''}$ such that $P \rightarrow Q$ is a chart for the morphism $X'' \rightarrow Y'$. In this language, given any fppf chart $P \rightarrow M_Y$,  a neat chart for a morphism $f$ extending $P$ always exists flat locally on $X$, but exists \'etale locally around a point $x \in X$ as well if the order of the torsion part of $\o{M}_{X/Y,\o{x}}$ is invertible at $x$ -- see \cite[III,Theorem 1.2.7]{oguslog}. \\

With these results at hand, we can state the following theorem:

\begin{theorem}
\label{theorem: Chart Criterion} Let $f:X \rightarrow Y$ be a morphism of locally Noetherian fine log schemes.
\begin{enumerate}
\item \label{existencechart} The map $f:X \rightarrow Y$ is log regular at $x \in X$, if and only if, for any fppf chart $P \rightarrow M_Y$, \'etale locally around $x$, there exists an injective chart $P \rightarrow Q$ for $f$ such the torsion part of the cokernel of $P^{\rm{gp}} \rightarrow Q^{\rm{gp}}$ has order invertible in $\mathcal{O}_{X,\overline{x}}$, and such that the morphism $X \rightarrow Y_P[Q]$ is regular.
\item \label{neatchart} If $f$ is log regular, then, given any fppf chart $P \rightarrow M_Y$ and any neat chart $P \rightarrow Q$ for $f$ \'etale locally on $X$, the morphism $X \rightarrow Y_P[Q]$ is regular.
\end{enumerate}
\end{theorem}

\begin{proof}
Suppose first that a chart $\phi: P \rightarrow Q$ as in the statement of condition (\ref{existencechart}) exists. Then, every map in the composition $X \rightarrow Y_P[Q] \rightarrow \c{Y}_P[Q] \rightarrow \L Y$ is regular at $x$ by Theorem \ref{coverth} and Lemma \ref{lem:log-smooth}, and thus the composition is also regular. Conversely, suppose $f$ is log regular at $x$. Assume first that $\overline{M}_{X/Y,\o{x}}^{\rm{gp}} \otimes k(\o{x}) = 0$. Then, the group $\overline{M}_{X/Y,\o{x}}^{\rm{gp}}$ is a finite group of order invertible in $\mathcal{O}_{X,\overline{x}}$, and we may take a neat chart $P \rightarrow Q$ for $f$ \'etale locally around $x$. As \'etale maps preserve regularity, and furthermore regularity can be checked \'etale locally, we may assume that the chart $P \rightarrow Q$ is defined globally. Then, the morphism $X \rightarrow Y$ is log regular, the morphism $Y_P[Q] \rightarrow Y$ is log \'etale, and hence $X \rightarrow Y_P[Q]$ is strict and log regular. It follows that $X \rightarrow Y_P[Q]$ is regular. If $\overline{M}_{X/Y,\o{x}}^{\rm{gp}} \otimes k(\o{x})$ is not $0$, take \'etale locally on $X$ a factorization $X \rightarrow Y':= Y \times \mathbb{A}^n \rightarrow Y$ as in Lemma \ref{lem: factorization}, which has $\overline{M}_{X/Y',\o{x}}^{\rm{gp}} \otimes k(\o{x}) = 0$. Take \'etale locally a neat chart $P \oplus \NN^n \rightarrow Q$ for $X \rightarrow Y'$, and look at the composed map $P \rightarrow Q$. By the snake lemma, there is a short exact sequence
\begin{align*}
\xymatrix{0 \ar[r] & \ZZ^n \ar[r] & Q ^{\rm{gp}}/P^{\rm{gp}}  \ar[r] & Q^{\rm{gp}}/(P^{\rm{gp}} \oplus \ZZ^n) \ar[r] & 0  }
\end{align*}
from which it follows that the torsion part of the cokernel $Q^{\rm{gp}} /P^{\rm{gp}}$ injects into $\overline{M}_{X/Y',\o{x}}^{\rm{gp}}$, and hence is invertible in $\mathcal{O}_{X,\overline{x}}$. We have already seen that the map $X \rightarrow Y'_{P \oplus \NN^n}[Q]$ is regular, as $\o{M}_{X/Y',\o{x}}\otimes k(\o{x}) = 0$; since $Y'_{P \oplus \NN^n}[Q] = Y_P[Q]$, the proof of condition (\ref{existencechart}) is complete. \\

To see that condition (\ref{neatchart}) holds, it suffices to notice that if $P \rightarrow Q$ is a neat chart for $f$ at $x$, then the map $P \oplus \NN^n \rightarrow M_X$ for the intermediate morphism $X \rightarrow Y'$  constructed in \ref{lem: factorization} can be chosen so that it lifts to a chart $P \oplus \NN^n \rightarrow Q$, which is also seen to be neat from the diagram

\begin{align*}
\xymatrix{0 \ar[r] & \ZZ^n \ar[d] \ar[r] & Q ^{\rm{gp}}/P^{\rm{gp}}  \ar[r] \ar[d] & Q^{\rm{gp}}/(P^{\rm{gp}} \oplus \ZZ^n) \ar[d] \ar[r] & 0  \\
 0 \ar[r] & \ZZ^n \ar[r] & \o{M}_{X,\o{x}}^{\gp}/\o{M}_{Y,f(\o{x})}^{\gp}  \ar[r] & \o{M}_{X,\o{x}}^{\rm{gp}}/(\o{M}_{Y,f(\o{x})}^{\rm{gp}} \oplus \ZZ^n) \ar[r] & 0  }
\end{align*}

As we have seen in the proof of (\ref{existencechart}) that the map $X \rightarrow Y_P[Q]$ constructed this way is regular, (\ref{neatchart}) follows.

\end{proof}

\begin{remark}
Condition (\ref{neatchart}) in Theorem \ref{theorem: Chart Criterion} is stronger than condition \ref{existencechart}: it asserts that the map $X \rightarrow Y_P[Q]$ is regular for \emph{any} \'etale local neat chart. However, note that it is not always possible to find such a chart. Indeed, a sufficient condition is that the torsion part of $\o{M}_{X/Y,\o{x}}$ has order invertible in $k(\o{x})$. For a log regular morphism, the condition is also necessary: if $f:X \rightarrow Y$ is log regular at $x$, and $P \rightarrow Q$ is an \'etale local chart for $f$, the map $X \rightarrow Y_P[Q]$ can never be regular if the cokernel of $P \rightarrow Q$ has torsion of order non-invertible at $x$: if it were, the map from $Y_P[Q] \rightarrow \L Y$ would be regular at the image of $x$ as well by \cite[$\rm{IV_2}$, 6.5.2(i)]{ega}, which contradicts \ref{lem:log-smooth}. Thus a log regular morphism $f$ for which $\o{M}_{X/Y,\o{x}}$ has torsion non-invertible in $k(\o{x})$ cannot have an \'etale local neat chart.
\end{remark}

The chart criterion easily implies a logarithmic version of Popescu's theorem.

\begin{corollary}
\label{cor: Popescu}
Let $f:X \rightarrow Y$ be a log regular morphism of locally Noetherian fine log schemes.
\begin{enumerate}
\item \'Etale locally around $x \in X$, $f$ factors as a composition of a log smooth morphism and a morphism which is log regular and strict.
\item \'Etale locally around $x \in X$, $f$ is the inverse limit of log smooth morphisms.
\end{enumerate}
\end{corollary}
\begin{proof}
By the chart criterion \ref{theorem: Chart Criterion}, $f$ locally factors as $X \rightarrow Y_P[Q] \rightarrow Y$ with $X \rightarrow Y_P[Q]$ strict and regular and $Y_P[Q] \rightarrow Y$ log smooth. Since the strict and regular map, that is, the underlying map of schemes $\u X \rightarrow \u{Y_P[Q]}$ is a filtered limit of smooth morphisms by Popescu's theorem (see \cite[Theorem~1.8]{Popescu} or \cite[Theorem~1.1]{Spiv}), the result follows.
\end{proof}

\end{section}

\subsection{Comparison to absolute log regularity}
Suppose now that $X$ is also saturated and log regular over a regular base $S=\Spec A$ with the trivial log structure. We further assume that the log structure $M_X$ is defined in the Zariski topology of $X$. As $X$ is saturated, $\o{M}_X$ is torsion free, and hence, we can find a neat chart $0 \rightarrow P$ for the structure morphism $X \rightarrow S$. Necessarily then $P$ is isomorphic to $\o{M}_{X,x}$, and, by the chart criterion \ref{theorem: Chart Criterion}, the map $X \rightarrow S[P]=\Spec A[P]$ is regular. Let $I$ denote the ideal generated by $M_{X,x}-\c{O}_{X,x}^*$ in $\c{O}_{X,x}$ and let $J$ be the ideal in $A[P]$ generated by $P-0$. Clearly, $I=J\c{O}_{X,x}$. As the morphism $\Spec(\c{O}_{X,x})\to S[P]$ is regular, it follows that $\Spec({O}_{X,x}/I)$ is regular over $\Spec(A[P]/J)=S$, and hence regular. Let $d$ be its dimension relative to $S$. We have
\begin{align*}
 \dim \c{O}_{X,x}  & = d + \dim A[P] = d + \dim A + \textup{rank}(P^\gp) = \\
& = \dim \c{O}_{X,x}/I + \textup{rank}(\o{M}_{X,x}^\gp)
\end{align*}

It follows that $X$ is log regular in the sense of Kato, \cite{Ktor}. Thus our notion of log regularity extends the absolute notion given in the literature to the relative case.

\bibliography{logregularrefs}
\end{document}